\documentclass[]{article}
\usepackage[]{tikz}
\usepackage[]{pgfplots}
\usepackage[]{amsthm}
\usepackage[]{amsmath}
\usetikzlibrary{shapes}
\pgfplotsset{compat=newest}
\usepackage[texencoding=ascii, style=alphabetic, maxbibnames=99]{biblatex}
\usepackage[]{hyperref}
\usepackage[]{geometry}

\newtheorem{theorem}{Theorem}
\newtheorem{proposition}{Proposition}
\theoremstyle{definition}
\newtheorem{definition}{Definition}

\let\set\mathbf

\title{Creating decidable diophantine equations}
\author{Robert DOUGHERTY-BLISS, Charles KENNEY, and Doron ZEILBERGER}
\date{\today}

\begin{document}

\maketitle

\centerline{\emph{Dedicated to Yuri Matiyasevich, a modern day Diophantus.}}

\begin{abstract}
    In 1974, 23-year-old Yuri Matiyasevich shattered Hilbert's dream to find a
    universal algorithm that would input an arbitrary polynomial of several
    variables with integer coefficients and determine whether it has integer
    solutions. It used, in a very clever way, sequences that satisfy
    second-order linear recurrences with integer coefficients. A side effect of
    his proof was the construction of infinite families of Diophantine
    equations for which the existence of solutions \emph{is} decidable, the
    so-called Pell equations.

    This was extended to higher-order (third and fourth) recurrences in a deep
    study by Matiyasevich's student, Maxim Vsemirnov, using sophisticated
    algebraic number theory. In the present, mostly expository, article, we
    revisit some of Vsemirnov's results from a more elementary viewpoint, and
    supplement it with a Maple implementation that would enable anyone to
    actually construct many families of decidable diophantine equations. Our
    method also can construct such diophantine equations for which one can
    prove that no solutions exist.
\end{abstract}

\section{Preface}

A diophantine equation is an equation of the form $P(x_1, x_2, \dots, x_n) = 0$
where $P$ is an integer coefficient polynomial, and we restrict the variables
$x_k$ to be integers. The most famous diophantine equation of all time is
\begin{equation*}
    x^2-2y^2=0.
\end{equation*}
Hippasus of Metapontum discovered that this equation has no solution in the
positive integers \cite{fritz}. Though Hippasus could not have phrased his
argument completely rigorously, he might have enjoyed something like this: Any
solution $(x, y)$ satisfies $x > y > 0$, and it is not hard to check that $(2y
- x, x - y)$ is also a positive solution with a smaller second coordinate.
(Note that $2y - x > x - y$ because $x \geq (3/2) y$ is impossible.) This is a
contradiction, because it would imply arbitrarily small positive integers.

Another equation where this kind of argument applies is {\it Pell's equation}
\begin{equation*}
    x^2-2 y^2 = \pm 1.
\end{equation*}
If $(x, y)$ is a solution, then so is $(x+2y,x+y)$. A simple brute force search
reveals the solution $(1, 1)$, and then this argument yields $(3, 2)$, $(7,
5)$, $(17, 12)$, and so on. A little more effort will show that \emph{all}
positive solutions come from repeatedly applying this map to $(1, 1)$
\cite{zeilbergerMotivated, pell}.

The previous two examples covered cases where a diophantine equation had
\emph{no} solutions, and one where it had \emph{infinitely many}. But both had
some subtle technical details that make it hard to imagine generalizing them to
other equations. In the early 1900's, David Hilbert believed that there should
be an argument that \emph{does} generalize to all diophantine equations. The
tenth problem on his famous 1900 list was to find it. But in 1970, 23-year-old
Yuri Matiyasevich, standing on the shoulders of Julia Robinson, Martin Davis,
and Hilary Putnam, shocked the world of mathematics by showing that Hilbert's
dream is impossible. There is no algorithm to determine whether an integer
coefficient polynomial has integer roots
\cite{juliarobinson,matiyasevich1993,mat71}.

Of course, for {\it specfic} equations, or even specific infinite families, one
can often decide the question. For example, Sir Andrew Wiles proved that
\begin{equation*}
    x^n+y^n=z^n
\end{equation*}
has \emph{no} positive integer solutions for $n > 2$, and Noam Elkies proved
that
\begin{equation*}
    A^4+B^4+C^4=D^4
\end{equation*}
has \emph{infinitely many} solutions \cite{wiles1995,elkies1988}.

A key fact in Matiyasevich's proof is that the solutions to the diophantine
equation
\begin{equation}
    \label{example-eqn}
    x^2 - b xy + y^2 = 1,
\end{equation}
where $x$ and $y$ are variables and $b$ is an integer ``parameter,'' are
described by the second-order linear recurrence
\begin{align*}
    a_b(0) &= 0,\ a_b(1) = 1 \\
    a_b(n+2) &= b \cdot a_b(n+1)-a_b(n).
\end{align*}
This is precise in the following sense: integers $x > y$ satisfy
\eqref{example-eqn} if and only if $(x, y) = (a_b(n + 1), a_b(n))$ for some
integer $n$ (\cite[pp.19-20]{matiyasevich1993} and \cite{jones1991}). The first
step towards this result begins with the matrix identity
\begin{equation*}
    \begin{bmatrix}
        a_b(n + 1) & -a_b(n) \\
        a_b(n) & -a_b(n - 1)
    \end{bmatrix}
    =
    \begin{bmatrix}
        b & -1 \\
        1 & 0
    \end{bmatrix}^n,
\end{equation*}
which is easily established by induction. Taking the determinant of both sides
and applying $a_b(n)$'s defining recurrence shows that consecutive terms of
$a_b(n)$ satisfy \eqref{example-eqn}. The converse relies on some slightly
technical arguments.

From a high level, Matiyasevich showed that a certain family of second order
recurrences were ``equivalent'' to a certain family of diophantine equations.
This was enough to resolve Hilbert's tenth problem, but it suggests a follow up
question: What about higher order recurrences? This question was brilliantly
investigated by M.~A.~Vsemirnov, a student of Matiyasevich, in the late 1990's
\cite{vsemirnovI,vsemirnovII}. Roughly speaking, Vsemirnov proved that only
certain recurrences up to order four can be equivalent to a diophantine
equation, and he determined \emph{all} such third order recurrences.
Vsemirnov's techniques rely on sophisticated algebraic number theory. Here, we
will show an elementary, constructive approach which relies only on calculus
and our computers.

\section{Diophantine equations from recurrences}

Let us first describe the construction which takes us from recurrences to
diophantine equations. Consider the linear recurrence
\begin{equation}
    \label{recdef}
    a(n) = c_1 a(n - 1) + \cdots + c_d a(n - d).
\end{equation}
Our recipe has two parts. First, the ``companion matrix''
\begin{equation*}
    B
    =
    \begin{bmatrix}
        c_1 & c_2 & \cdots & c_d \\
        1   &   0 & \cdots & 0   \\
        0   &   1 & \cdots & 0 \\
            & & \cdots\\
        0   &   0 & \cdots \ 1 & 0
    \end{bmatrix}
\end{equation*}
satisfies the forward identity
\begin{equation*}
    \begin{bmatrix}
        a(n + 1) \\
        a(n) \\
        \vdots \\
        a(n - d + 2)
    \end{bmatrix}
    =
    B
    \begin{bmatrix}
        a(n) \\
        a(n - 1) \\
        \vdots \\
        a(n - d + 1)
    \end{bmatrix}
\end{equation*}
for any sequence which satisfies \eqref{recdef}. Second, \eqref{recdef} has
$d$ ``fundamental solutions'' which form a basis for all solutions
\cite[ch.~4]{concrete}. They are the solutions whose initial conditions (the
values $(a(0), a(1), \dots, a(d - 1))$) are all zero except for a single entry,
which is instead one. We will call these solutions $e_k(n)$, and give them the
following initial conditions:
\begin{equation*}
    \begin{bmatrix}
        1 & 0 & 0 & \cdots & 0 \\
        0 & 1 & 0 & \cdots & 0 \\
        0 & 0 & 1 & \cdots & 0 \\
          &   &   & \cdots &   \\
        0 & 0 & 0 & \cdots & 1 \\
    \end{bmatrix}
    =
    \begin{bmatrix}
        e_0(d - 1) & e_1(d - 1) & \cdots & e_{d - 1}(d - 1) \\
        e_0(d - 2) & e_1(d - 2) & \cdots & e_{d - 1}(d - 2) \\
        e_0(d - 3) & e_1(d - 3) & \cdots & e_{d - 1}(d - 3) \\
          &   \cdots &   \\
        e_0(0) & e_1(0) & \cdots & e_{d - 1}(0) \\
    \end{bmatrix}.
\end{equation*}
This implies the relation
\begin{equation}
    \label{matident}
    B^n = B^n I =
    \begin{bmatrix}
        e_0(n + d - 1) & e_1(n + d - 1) & \cdots & e_{d - 1}(n + d - 1) \\
        e_0(n + d - 2) & e_1(n + d - 2) & \cdots & e_{d - 1}(n + d - 2) \\
        e_0(n + d - 3) & e_1(n + d - 3) & \cdots & e_{d - 1}(n + d - 3) \\
          &   \cdots &   \\
        e_0(n) & e_1(n) & \cdots & e_{d - 1}(n) \\
    \end{bmatrix}.
\end{equation}
From the elementary theory of difference equations (again see
\cite[ch.~4]{concrete}), every solution to \eqref{recdef}---including the
fundamental ones---can be expressed as a linear combination of the sequences
$e_0(n)$, $e_0(n + 1)$, \dots, $e_0(n + d - 1)$. Therefore every entry in the
right-hand side of \eqref{matident} is actually a linear combination of shifts
of $e_0(n)$. Taking determinants in \eqref{matident} yields
\begin{equation*}
    P(e_0(n), e_0(n + 1), \dots, e_0(n + d - 1)) = (\det B)^n
\end{equation*}
for some polynomial $P$, and Laplace expansion on the first row gives $\det B =
(-1)^{d + 1} c_d$. If we take $c_d = (-1)^{d + 1}$ then the right-hand side is
$1$.

These considerations lead to the following proposition.

\begin{proposition}
    \label{existence}
    For any integers $c_1, c_2, \dots, c_{d - 1}$, not all zero, there is a
    nonzero polynomial $P(x_1, x_2, \dots, x_d)$ such that the diophantine
    equation
    \begin{equation*}
        P(x_1, x_2, \dots, x_d) = 1
    \end{equation*}
    has infinitely many solutions. In particular, the points $(a(n), a(n + 1),
    \dots, a(n + d - 1))$ are solutions, where $a(n)$ satisfies
    \begin{equation*}
        a(n) = \sum_{k = 1}^{d - 1} c_k a(n - k) + (-1)^{d + 1} a(n - d)
    \end{equation*}
    and has initial conditions $0, 0, \dots, 0, 1$.
\end{proposition}

This is roughly half of Matiyasevich's proof characterizing solutions to
\eqref{example-eqn}. The next step is to generalize the converse, and show that
the diophantine equations in Proposition~\ref{existence} are sometimes solved
by \emph{only} the recurrence solutions. This goal is too lofty in general, but
we have arguments which apply to an infinite family of recurrences, and one
detailed case study concerning the Tribonacci numbers.

\section{Tribonacci numbers}

\begin{definition}
    Define the numbers $T_n$ by
    \begin{align*}
        T_0 &= T_1 = 0 \\
        T_2 &= 1 \\
        T_n &= T_{n - 1} + T_{n - 2} + T_{n - 3},
    \end{align*}
    the polynomial $P_T$ by
    \begin{equation*}
        P_T(x, y, z) = x^3+2x^2y+x^2z+2xy^2-2xyz-xz^2+2y^3-2yz^2+z^3,
    \end{equation*}
    and the map $R_T$ by
    \begin{equation*}
        R_T(x, y, z) = (y, z, x + y + z).
    \end{equation*}
    Note that $P_T$ is invariant under $R_T$, i.e., $P_T \circ R_T = P_T$.
\end{definition}

\begin{proposition}
    \label{trib-reduction}
    If $P_T(x, y, z) = 1$ for integers $(x, y, z)$, then $(x, y, z)$ is the
    result of repeatedly applying $R_T$ or its inverse to a nonnegative
    increasing solution. That is, there exist integers $0 \leq a < b < c$ and a
    positive integer $n$ such that $P_T(a, b, c) = 1$ and $(x, y, z)$ is
    $R_T^n(a, b, c)$ or $R_T^{-n}(a, b, c)$.
\end{proposition}

\begin{proof}
    Repeatedly applying $R_T$ to our initial point $(x, y, z)$ produces a
    sequence $a(n)$ which satisfies
    \begin{equation*}
        a(n) = a(n - 1) + a(n - 2) + a(n - 3)
    \end{equation*}
    with initial conditions $(a(0), a(1), a(2)) = (x, y, z)$. Because $P_T$ is
    invariant under $R_T$ we have $P_T(a(n), a(n + 1), a(n + 2)) = 1$ for all
    $n$. The elementary theory of difference equations implies $a(n) \sim
    c \cdot \alpha^n$ where $\alpha \approx 1.8393$ is the unique real root of $X^3 - X^2 - X
    - 1$ and
    \begin{equation*}
        c = \alpha \frac{(\alpha^2 - \alpha - 1) a(0) + (\alpha - 1) a(1) + a(2)}{\alpha^2 + 2 \alpha + 3}.
    \end{equation*}

    Note that $c$ is real. If $c < 0$, then we eventually obtain a strictly
    negative solution, which is impossible because $P_T(x, y, z) \leq 0$ if $x,
    y, z \leq 0$. If $c = 0$ then $(\alpha^2 - \alpha - 1) a(0) + (\alpha - 1)
    a(1) + a(2) = 0$, and this is impossible because $\{1, \alpha, \alpha^2\}$
    is linearly independent over the rationals. (The minimal polynomial is a
    cubic and irreducible over $\set{Q}$, so $\alpha$ is a degree three
    algebraic integer.) The remaining possibility is $c
    > 0$, which implies that we eventually have $0 < a(n) < a(n + 1) < a(n +
    2)$. We get back to $(x, y, z)$ by repeatedly applying the inverse map
    $R_T^{-1}$.
\end{proof}

\begin{proposition}
    \label{trib-inc}
    If $P_T(x, y, z) = 1$ for integers $0 < x < y < z$, then $(x, y, z) = (T_n,
    T_{n + 1}, T_{n + 2})$ for some integer $n \geq 0$.
\end{proposition}

\begin{proof}
    The map $R_T^{-1}(x, y, z) = (z - x - y, x, y)$ takes solutions to other
    solutions. Note that if $0 < z - x - y < x$, then the new solution is also
    positive and increasing, and in fact strictly smaller in magnitude. We will
    show that $0 < z - x - y < x$ for all increasing solutions with
    sufficiently large $z$.

    If we divide both sides of the equation $P_T(x, y, z) = 1$ by $z^3$, and
    make the change of variables $(t, s) = (x / z, y / z)$, then we obtain
    \begin{equation}
        \label{trib-cubic}
        2s^3+2s^2t+2st^2+t^3-2st+t^2-2s-t+1 = \frac{1}{z^3}.
    \end{equation}
    Call the left-hand side of this equation $f(s, t)$ and note that it is a
    cubic defined on the unit square. It is a routine (computer-assisted)
    calculus exercise to show that the minimum of $f(s, t)$ on the region $1 -
    t - s \leq 0$ is
    \begin{equation*}
        \frac{398 - 68 \sqrt{34}}{27} > 0.
    \end{equation*}
    Therefore we cannot have both \eqref{trib-cubic} and $1 - t - s \leq 0$ for
    \begin{equation*}
        z > \left( \frac{398 - 68 \sqrt{34}}{27} \right)^{-1/3} \approx 2.6235.
    \end{equation*}
    It follows that $0 < z - x - y$ for all increasing solutions to $P_T(x,
    y, z) = 1$ with $z \geq 3$. By an analogous argument on the region $1 - t -
    s \geq t$, all increasing solutions to $P_T(x, y, z) = 1$ with $z \geq 5$
    satisfy $z - x - y < x$.

    Repeatedly applying the ``backwards'' map $R_T^{-1}$ produces smaller,
    positive, increasing solutions as long as $z \geq 5$, and so this
    process terminates at a solution with $0 < x < y < z < 5$. It is
    simple to check that all such solutions return to the point $(0, 0, 1)$
    under the map $R_T^{-1}$, and so \emph{all} increasing positive
    solutions come from applying the ``forward'' map $R_T$ to $(0, 0, 1)$. This
    produces exactly the Tribonacci numbers.
\end{proof}

See Figure~\ref{plane-map} for a visual representation of the maps and regions
in Proposition~\ref{trib-inc}.

\begin{theorem}
    \label{trib-all}
    If $P_T(x, y, z) = 1$ for integers $x, y, z$, then $(x, y, z) = (T_n, T_{n
    + 1}, T_{n + 2})$ for some integer $n$.
\end{theorem}

\begin{proof}
    By the previous two propositions, every solution comes from applying the
    maps $(x, y, z) \mapsto (y, z, x + y + z)$ and $(x, y, z) \mapsto (z - x -
    y, x, y)$ to the solution $(0, 0, 1)$, which produces exactly the
    Tribonacci numbers with positive and negative indices.
\end{proof}

\begin{figure}[htpb]
    \begin{minipage}{0.5\textwidth}
        \begin{tikzpicture}[]
            \begin{axis}[
                xmin=0, xmax=1,
                ymin=0, ymax=1,
                zmin=0, zmax=1,
                view={0}{90},
                axis equal image]
            \draw(0,0)--(1, 0)--(1,1)--(0,1)--(0,0);
            \draw[fill=gray!42](0,1)--(1,0)--(1,1)--(0,1);
            \draw[fill=gray!42](0,1)--(0.5,0)--(0,0)--(0,1);
            \filldraw (0.2956, 0.5437) circle (2pt);
            \addplot3[black!90, quiver={u=y/(x+y+1)-x, v=1/(x+y+1)-y, scale arrows=0.15}, samples=40, -latex] (x, y, 0);
            \end{axis}
        \end{tikzpicture}
    \end{minipage}
    \begin{minipage}{0.5\textwidth}
        \begin{tikzpicture}[]
            \begin{axis}[
                xmin=-3, xmax=3,
                ymin=-3, ymax=3,
                view={0}{90},
                axis equal image]
                \draw [thick] (0, 0)--(1, 0)--(1, 1)--(0, 1)--(0, 0);
                \draw [dashed, thick] (-3, 2)--(2,-3);
            \filldraw (0.2956, 0.5437) circle (3pt);
            \addplot3[
                black!90,
                quiver={u=y/(x+y+1)-x, v=1/(x+y+1)-y, scale arrows=0.10},
                samples=27,
                x filter/.expression={abs(x + y + 1) < 0.4 ? nan : x},
                -latex] (x, y, 0);
            \end{axis}
        \end{tikzpicture}
    \end{minipage}
    \caption[Tribonacci iteration map transformed to the plane]{The map $(x, y, z) \mapsto (y, z, x + y + z)$ represented in the $ts$
    plane by its equivalent $(t, s) \mapsto (s / (s + t + 1), 1 / (s + t +
    1))$. Left: The map restricted to the unit square, with the region $\{s + t
    > 1\} \cup \{s + 2t < 1\}$ shaded. The unique critical point in the first
    quadrant of the left-hand side of \eqref{trib-cubic} is labeled by a black
    dot. Right: The map on a larger portion of the plane. The critical point is
    an attractor for $s + t + 1 > 0$.}
    \label{plane-map}
\end{figure}
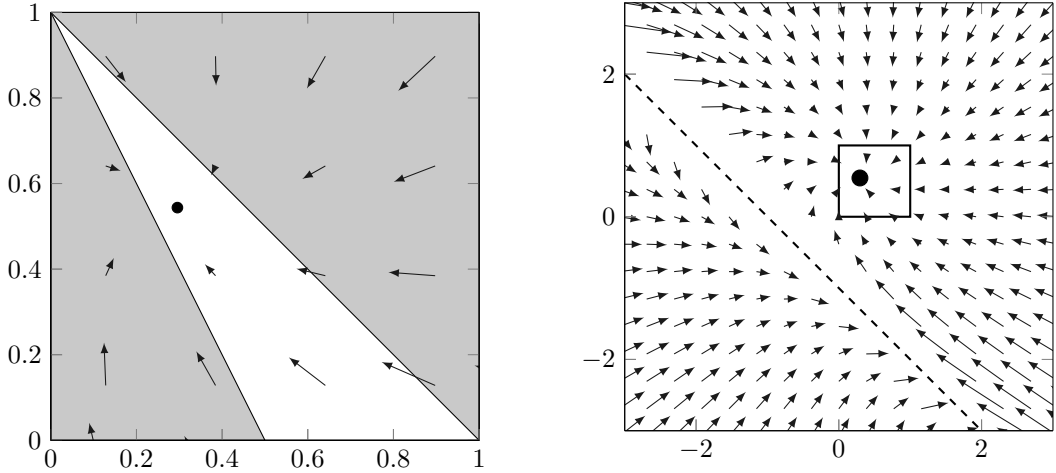

\section{Uniqueness in general}

The arguments from the previous section carry over almost verbatim to the
general third-order recurrence. The main difficulty is in establishing the
minimum of the analogous cubic \eqref{trib-cubic}. For any specific recurrence
it is completely routine to check whether the proof of
Proposition~\ref{trib-inc} works, but Proposition~\ref{gen-inc} gives a weaker
statement about an infinite family.

\begin{definition}
    For any positive integers $a$ and $b$, define the polynomial $P_{ab}(x, y,
    z)$ as
    \begin{equation*}
        a^2y^2z+abxyz+aby^3+b^2xy^2+ax^2z+axy^2-2ayz^2+2bx^2y-bxz^2-by^2z+x^3-3xyz+y^3+z^3
    \end{equation*}
\end{definition}

\begin{proposition}
    \label{gen-reduction}
    Let $a$ and $b$ be positive integers such that $X^3 - a X^2 - bX - 1$ is
    irreducible over $\set{Q}$ and has a single largest root which is real and
    greater than $1$. Then all integer solutions to $P_{ab}(x, y, z) = 1$ are
    generated by applying the map $(x, y, z) \mapsto (z - ay - bx, x, y)$ or
    its inverse to a positive, increasing solution.
\end{proposition}

\begin{proof}
    The argument is the same as in Proposition~\ref{trib-reduction}. The
    irreducibility of $X^3 - a X^2 - bX - 1$, with largest root $\alpha > 1$,
    implies the linear independence of $\{1, \alpha, \alpha^2\}$ over the
    rationals and gives the correct asymptotics.
\end{proof}

\begin{proposition}
    \label{gen-inc}
    Fix positive integers $a$ and $b$ and consider the recurrence
    \begin{equation}
        \label{3-rec}
        u(n) = a u(n - 1) + b u(n - 2) + u(n - 3).
    \end{equation}
    If $a$ is sufficiently large relative to $b$, then all solutions $0 < x < y
    < z$ to the diophantine equation $P_{ab}(x, y, z) = 1$ are generated by
    applying \eqref{3-rec} to finitely many solutions.
\end{proposition}

It should be noted that while the following proof is non-constructive, the
\emph{method} is not. Carrying out the proof for any \emph{specific} integers
$a$ and $b$ will determine an exact bound under which the finitely many initial
conditions can be found.

\begin{proof}
    The polynomial $P_{ab}(x, y, z)$ is invariant under the map
    \begin{equation}
        \label{backwards-gen}
        (x, y, z) \mapsto (z - ay - bx, x, y),
    \end{equation}
    so it takes solutions to solutions. Note that
    \begin{align*}
        |(z - ay - bx, x, y)|^2 - |(x, y, z)|^2
        =
        (ay + bx)(ay + bx - 2z) < 0
    \end{align*}
    provided that the conditions $x, y \geq 0$ and
    \begin{equation}
        \label{gen-condition}
        \quad ay + bx - 2z < 0
    \end{equation}
    are satisfied, meaning that solutions are taken to smaller solutions in
    this case. We will show that, given an increasing solution $0 \leq x < y <
    z$ with sufficiently large $z$, condition \eqref{gen-condition} holds, and
    that the new solution determined by \eqref{backwards-gen} is also
    nonnegative, increasing and satisfies \eqref{gen-condition}.

    If we divide both sides of $P_{ab}(x, y, z) = 1$ by $z^3$ and make the
    change of variables $(t, s) = (x / z, y / z)$, then we obtain $f_{ab}(t, s)
    = z^{-3}$ where $f_{ab}$ is a cubic in $t$ and $s$ on the unit square.
    The region that we
    wish to avoid is
    \begin{equation*}
        R_{ab} = \{as + bt \geq 2 \} \cup \{1 - as - bt \leq 0\} \cup \{1 - as - bt \geq t\},
    \end{equation*}
    where we implicitly are restricting everything to the unit square. Note
    that the first set in the union is contained in third set if $b \geq 1$, so
    our region is really just
    \begin{equation*}
        R_{ab} = \{1 - as - bt \leq 0\} \cup \{1 - as - bt \geq t\}.
    \end{equation*}
    Because $f_{ab}$ is a cubic, it is possible to exactly compute its critical
    points on the unit square, as well as the critical points of boundary
    functions such as $f_{ab}(0, s)$ and $f_{ab}(1, s)$. If we treat $b$ as a
    constant and perform asymptotic expansions as $a \to \infty$ of these
    critical points, it turns out that the minimum of $f_{ab}$ on the region
    $R_{ab}$ occurs on the line $1 - as - bt = 0$, and it equals
    \begin{equation*}
        \frac{1}{a^6} - \frac{b^2}{4 a^7} - \frac{9b}{2a^8} + O(a^{-9}).
    \end{equation*}
    So $f_{ab}(t, s) = z^{-3}$ fails if
    \begin{equation}
        \label{asy-bound}
        z > a^2 + \frac{b^2}{12} a + \frac{3b}{2} + \frac{b^4}{72} + O(a^{-1}).
    \end{equation}
    It follows that the inequalities
    \begin{align*}
        0 < 1 - as - bt < t &\iff 0 < z - ay - bx < x \\
        as + bt < 2 &\iff ay + bx < 2z
    \end{align*}
    hold for any solution $0 < x < y < z$ with sufficiently large $z$. We may
    therefore iterate \eqref{backwards-gen} on such a solution until we reach
    one where $z$ is below the bound implied by \eqref{asy-bound}, and there
    are only finitely many of these.
\end{proof}

\begin{theorem}
    \label{gen-soln}
    Let $a$ and $b$ be positive integers such that
    \begin{enumerate}
        \item $X^3 - a X^2 - bX - 1$ is irreducible over the rationals and has
            a single largest root which is real and greater than $1$; and
        \item $a$ is sufficiently large relative to $b$ (in the
            non-constructive sense of proposition~\ref{gen-reduction}).
    \end{enumerate}
    Then all integer solutions to $P_{ab}(x, y, z) = 1$ are generated by
    applying \eqref{3-rec} forwards or backwards to finitely many initial
    solutions.
\end{theorem}

Note that the first condition is not very restrictive. The cubic $X^3 - a X^2 -
bX - 1$ has a rational root (by the rational root test) only if $b = a + 2$.

\section{Examples}

\paragraph{A single family} The characteristic equation
\begin{equation*}
    X^3 - 10 X^2 - 3 X - 1
\end{equation*}
leads to the diophantine equation
\begin{equation*}
    x^3+6x^2y+10x^2z+19xy^2+27xyz-3xz^2+31y^3+97y^2z-20yz^2+z^3 = 1.
\end{equation*}
Theorem~\ref{gen-soln} (along with explicit arguments from
Proposition~\ref{gen-inc}) shows that all solutions to this equation are
generated by applying the maps $(x, y, z) \mapsto (y, z, 10z + 3y + x)$ and $(x,
y, z) \mapsto (z - 10y - 3x, x, y)$ to the
initial solution $(0, 0, 1)$.

\paragraph{Multiple families} The characteristic equation
\begin{equation*}
    X^3 - 2 X^2 - 3 X - 1
\end{equation*}
leads to the diophantine equation
\begin{equation*}
    x^3+6x^2y+2x^2z+11xy^2+3xyz-3xz^2+7y^3+y^2z-4yz^2+z^3 = 1.
\end{equation*}
Theorem~\ref{gen-soln} (along with explicit arguments from
Proposition~\ref{gen-inc}) shows that all solutions to this equation are
generated by applying the maps $(x, y, z) \mapsto (y, z, 3z + 2y + x)$ and $(x,
y, z) \mapsto (z - 3y - 2x, x, y)$ to the
initial solutions
\begin{equation*}
    (0, 0, 1), (0, 1, 3), (0, 2, 7), (1, 1, 4),
\end{equation*}
none of which can be obtained from any other.

See Figure~\ref{plane-map-other} for a visual demonstration of these maps.

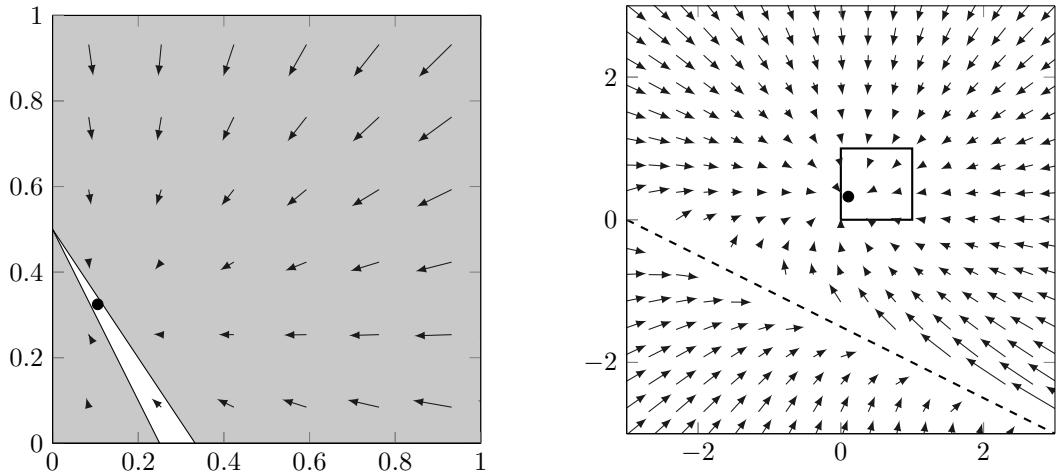
\begin{figure}[htpb]
    \begin{minipage}{0.5\textwidth}
        \begin{tikzpicture}[]
            \begin{axis}[
                xmin=0, xmax=1,
                ymin=0, ymax=1,
                zmin=0, zmax=1,
                view={0}{90},
                axis equal image]
            \draw(0,0)--(1, 0)--(1,1)--(0,1)--(0,0);
            \draw[fill=gray!42](0,0.5)--(0.333,0)--(1,0)--(1,1)--(0,1);
            \draw[fill=gray!42](0,0.5)--(0.25,0)--(0,0)--(0,0.5);
            \filldraw (0.1054, 0.3247) circle (2pt);
            \addplot3[black!90, quiver={u=y/(2 * y + x + 3)-x, v=1/(2 * y + x + 3)-y, scale arrows=0.10}, samples=60, -latex] (x, y, 0);
            \end{axis}
        \end{tikzpicture}
    \end{minipage}
    \begin{minipage}{0.5\textwidth}
        \begin{tikzpicture}[]
            \begin{axis}[
                xmin=-3, xmax=3,
                ymin=-3, ymax=3,
                view={0}{90},
                axis equal image]
                \draw [thick] (0, 0)--(1, 0)--(1, 1)--(0, 1)--(0, 0);
                \draw [dashed, thick] (-3, 0)--(3,-3);
            \filldraw (0.1054, 0.3247) circle (2pt);
            \addplot3[
                black!90,
                quiver={u=y/(2*y + x + 3)-x, v=1/(2*y + x + 3)-y, scale arrows=0.10},
                samples=27,
                x filter/.expression={abs(2 * y + x + 3) < 0.5 ? nan : x},
                -latex] (x, y, 0);
            \end{axis}
        \end{tikzpicture}
    \end{minipage}
    \caption[Tribonacci iteration map transformed to the plane]{The map $(x, y,
        z) \mapsto (y, z, 3z + 2y + x)$ represented in the $ts$ plane by its
        equivalent $(t, s) \mapsto (s / (2s + t + 3), 1 / (2 s + t +
    3))$. Left: The map restricted to the unit square, with the region $\{2s +
    3t \geq 1\} \cup \{2s + 4t \leq 1\}$ shaded. The unique critical point
    in the first quadrant of the cubic constructed in the proof of
    Proposition~\ref{gen-inc} with parameters $(a, b) = (2, 3)$ is labeled
by a black dot. Right: The map on a larger portion of the plane.}
\label{plane-map-other}
\end{figure}

\paragraph{A failure} The characteristic equation
\begin{equation*}
    X^3 - X^2 - 3X - 1 = (X - 1)(X^2 - 2X - 1)
\end{equation*}
corresponds to setting $a = 1$ and $b = 3$, which leads to the diophantine
equation
\begin{equation*}
    x^3+6x^2y+x^2z+10xy^2-3xz^2+4y^3-2y^2z-2yz^2+z^3 = 1.
\end{equation*}
Our method fails here on two counts. First, the proof of Proposition~\ref{gen-inc}
does not go through ($a = 1$ is not big enough relative to $b = 3$). Second,
this recurrence has degenerate integer solutions like $(-1)^n$ which do not
have the correct asymptotics.

\section{Computer implementations}

As mentioned before the proof of Proposition~\ref{gen-inc}, most of our arguments
here can be made effective for any set of fixed parameters. The authors have
written a Maple package, \texttt{Hilbert10.txt}, available at \url{https://sites.math.rutgers.edu/~zeilberg/mamarim/mamarimhtml/hilbert10.html}. Execute the command \texttt{ezra();} to receive a help display.

To execute the proof of Proposition~\ref{gen-inc} on a specific recurrence, use
the command \texttt{allSolns([a, b, c])}. The procedure requires a list of
length three with $c = 1$. This is slightly awkward, but we leave it as is to
suggest the challenge of generalizing it.

The following example computes the Diophantine equation induced by the
Tribonacci recurrence and finds all of its solutions:
\begin{verbatim}
> allSolns([1, 1, 1]);
findAbsoluteMin:   making a floating point guess:   -2. <= 0
allSolns:   looking for monotonically increasing solutions up to   5
                   {[0, 0, 1]}
\end{verbatim}
The output states that all solutions are generated by applying the Tribonacci
recurrence to the initial solution $(0, 0, 1)$. Note that \texttt{allSolns}
searches (without loss of generality) for \emph{monotonically} increasing
solutions.

The following example does the same thing for a different third-order
recurrence:
\begin{verbatim}
> allSolns([5, 3, 1]);
findAbsoluteMin:   making a floating point guess:   -7.333333333 <= 0
allSolns:   looking for monotonically increasing solutions up to   36
                   {[0, 0, 1]}

\end{verbatim}

And finally, an example with several generating initial conditions:
\begin{verbatim}
> allSolns([2, 3, 1]);
findAbsoluteMin:   making a floating point guess:   -2. <= 0
allSolns:   looking for monotonically increasing solutions up to   17
           {[0, 0, 1], [0, 1, 3], [0, 2, 5], [1, 1, 4]}

\end{verbatim}

The package also includes a verbose ``proof printer,'' \texttt{verboseProof},
which fills in the details of Proposition~\ref{gen-inc} for specific
parameters.

\begin{verbatim}
> verboseProof([2, 3, 1]);
findAbsoluteMin:   making a floating point guess:   -2. <= 0

THEOREM. The nonnegative, increasing solutions to the
    diophantine equation

 3      2        2           2                  2      3
x  + 6 x  y + 2 x  z + 11 x y  + 3 x y z - 3 x z  + 7 y

        2          2    3
     + y  z - 4 y z  + z  = 1

          are generated by applying the recurrence

                         [2, 3, 1]

            to finitely many initial solutions.

  PROOF. Let P be the polynomial P on the left-hand side.

      Note that it is invariant under the recurrence:

                     P - P(shift) is, 0

The backwards shift formula to get the previous term
    from the triple (x, y, z) is

                       z - 3 x - 2 y

We will show that this backwards shift gives a smaller
    increasing solution for sufficiently large z.

                                       3
Divide both sides of our equation by, z ,

    and make the change of variables {t = x / z, s = y / z}

                         This gives

   3       2          2    3    2              2
7 s  + 11 s  t + 6 s t  + t  + s  + 3 s t + 2 t  - 4 s

                  1
     - 3 t + 1 = ----
                   3
                  z

            where (t, s) is in the unit square.

Let (x, y, z) be a generic solution. Then the following
    inequalities are routine calculus exercises.

the inequality, 0 < z - 3 x - 2 y, also known as,

    0 < -2 s - 3 t + 1, holds for

                          1
              -------------------------- <= z
              /                1/2\(1/3)
              |50371   1718 859   |
              |----- - -----------|
              \81675      81675   /

                    more explicitly for

                      16.36065936 <= z

the inequality, z - 3 x - 2 y < x, also known as,

    -2 s - 3 t + 1 < t, holds for

                           1
                ----------------------- <= z
                /             1/2\(1/3)
                |161   53 1219   |
                |--- - ----------|
                \216      2484   /

                    more explicitly for

                      13.33123227 <= z

  We only need to look for solutions with, z < 16.36065936

           and there are finitely many of these.

                           Q.E.D.
\end{verbatim}

\section*{Acknowledgements}
Many thanks are due to Yuri Matiyasevich for telling us
about the deep work of Maxim Vsemirnov.

\printbibliography[heading=bibintoc]

\end{document}